\documentclass[10pt,letterpaper]{amsart}

\usepackage{amsmath}
\usepackage{amssymb}
\usepackage{amsthm}
\usepackage{amsfonts}
\usepackage{microtype}
\usepackage{mathrsfs}
\usepackage{graphicx}
\usepackage{color}
\usepackage{latexsym}
\usepackage{rotating}
\usepackage{xspace}
\usepackage[all]{xy}
\usepackage{longtable}
\usepackage{leftidx}
\usepackage{mathtools}
\usepackage{titlesec}
\usepackage{setspace}
\usepackage{tikz}
\usepackage{tikz-cd}
\usetikzlibrary{calc}
\usetikzlibrary{arrows}
\usetikzlibrary{decorations.markings,intersections}
\usepackage[top=1.2in, bottom=1.2in]{geometry}
\usepackage{geometry}
\usepackage[final]{pdfpages}
\usepackage{fancyhdr}

\newtheorem{theorem}{Theorem}[section]
\numberwithin{equation}{theorem}
\newtheorem*{theorem*}{Theorem}
\newtheorem{lemma}[theorem]{Lemma}

\newtheorem{corollary}[theorem]{Corollary}

\theoremstyle{definition}

\newtheorem{remark}[theorem]{Remark}

\theoremstyle{conjecture}

\newcommand{\Ass}{\operatorname{Ass}}
\newcommand{\Ker}{\operatorname{Ker}}
\newcommand{\Coker}{\operatorname{Coker}}
\newcommand{\im}{\operatorname{Im}}

\newcommand{\ann}{\operatorname{ann}}
\newcommand{\rank}{\operatorname{rank}}

\newcommand{\id}{\operatorname{id}}
\newcommand{\fd}{\operatorname{fd}}

\newcommand{\pd}{\operatorname{pd}}

\newcommand{\Ext}{\operatorname{Ext}}

\newcommand{\Tor}{\operatorname{Tor}}
\newcommand{\Hom}{\operatorname{Hom}}

\newcommand{\edim}{\operatorname{edim}}
\newcommand{\depth}{\operatorname{depth}}

\newcommand{\type}{\operatorname{type}}
\newcommand{\Tot}{\operatorname{Tot}}

\newcommand{\suchthat}{\;\ifnum\currentgrouptype=16 \middle\fi|\;}

\makeatletter
\newcommand{\hocolim@}[2]{%
  \vtop{\m@th\ialign{##\cr
    \hfil$#1\operator@font holim$\hfil\cr
    \noalign{\nointerlineskip\kern1.5\ex@}#2\cr
    \noalign{\nointerlineskip\kern-\ex@}\cr}}%
}
\newcommand{\hocolim}{%
  \mathop{\mathpalette\hocolim@{\rightarrowfill@\textstyle}}\nmlimits@
}
\makeatother

\makeatletter
\newcommand{\holim@}[2]{%
  \vtop{\m@th\ialign{##\cr
    \hfil$#1\operator@font holim$\hfil\cr
    \noalign{\nointerlineskip\kern1.5\ex@}#2\cr
    \noalign{\nointerlineskip\kern-\ex@}\cr}}%
}
\newcommand{\holim}{%
  \mathop{\mathpalette\holim@{\leftarrowfill@\textstyle}}\nmlimits@
}
\makeatother

\makeatletter
\def\@secnumfont{\bfseries}
\def\section{\@startsection{section}{1}%
  \z@{.7\linespacing\@plus\linespacing}{.5\linespacing}%
  {\normalfont\Large\bfseries\filcenter}}
\def\subsection{\@startsection{subsection}{2}%
  \z@{.5\linespacing\@plus.7\linespacing}{-.5em}%
  {\normalfont\large\bfseries}}
\makeatother

\DeclareFontFamily{OT1}{pzc}{}
\DeclareFontShape{OT1}{pzc}{m}{it}{<-> s * [1.20] pzcmi7t}{}
\DeclareMathAlphabet{\mathpzc}{OT1}{pzc}{m}{it}

\makeatletter
\def\moverlay{\mathpalette\mov@rlay}
\def\mov@rlay#1#2{\leavevmode\vtop{%
   \baselineskip\z@skip \lineskiplimit-\maxdimen
   \ialign{\hfil$\m@th#1##$\hfil\cr#2\crcr}}}
\newcommand{\charfusion}[3][\mathord]{
    #1{\ifx#1\mathop\vphantom{#2}\fi
        \mathpalette\mov@rlay{#2\cr#3}
      }
    \ifx#1\mathop\expandafter\displaylimits\fi}
\makeatother

\makeatletter
\providecommand{\bigsqcap}{%
  \mathop{%
    \mathpalette\@updown\bigsqcup
  }%
}
\newcommand*{\@updown}[2]{%
  \rotatebox[origin=c]{180}{$\m@th#1#2$}%
}
\makeatother

\begin{document}

\author[Hossein Faridian]{Hossein Faridian}

\title[Bounds on Injective Dimension and Exceptional Complete Intersection Maps]
{Bounds on Injective Dimension and \\ Exceptional Complete Intersection Maps}

\address{Hossein Faridian, School of Mathematical and Statistical Sciences, Clemson University, Clemson, SC 29634, USA.}
\email{hfaridi@g.clemson.edu}

\subjclass[2010]{13D05; 13C11; 13J99; 13D03.}

\keywords {injective dimension; exceptional complete intersection map; Andr\'{e}-Quillen homology}

\begin{abstract}
We prove that if $f:R \rightarrow S$ is a local homomorphism of noetherian local rings, and $M$ is a non-zero finitely generated or artinian $S$-module whose injective dimension over $R$ is bounded by the difference of the embedding dimensions of $R$ and $S$, then $M$ is an injective $S$-module and $f$ is an exceptional complete intersection map.
\end{abstract}

\maketitle

\sloppy

\section{Introduction}

In the past fifty years or so, a great body of research has been accomplished by algebraists on the properties of rings in terms of those of their modules. The trend has gradually shifted towards the relative situation in which one studies the properties of ring homomorphisms in terms of those of the modules along them. The focus of this work is on surjective homomorphisms of noetherian local rings. Such a map is a complete intersection map if its kernel is generated by a regular sequence. If in addition, the regular sequence is part of a minimal generating set for the maximal ideal of the source, then the homomorphism is called an exceptional complete intersection map. This class of maps is of particular interest as they arise naturally in various contexts; for example, the diagonal of a smooth map is locally exceptional complete intersection.

It is proved in \cite{BIK} that certain homological invariants of modules over ring homomorphisms detect the exceptional complete intersection property. More specifically, it is shown that if $f:R \rightarrow S$ is a local homomorphism of noetherian local rings, and $M$ is a non-zero finitely generated $S$-module with $\fd_{R}(M)\leq \edim(R)- \edim(S)$, then $M$ is a free $S$-module, $f$ is an exceptional complete intersection map, and $\fd_{R}(M)= \edim(R)- \edim(S)= \fd_{R}(S)$; see \cite[Theorem 3.1]{BIK}. The goal of this article is to prove an analogue of this theorem involving the injective dimension. More specifically, we prove the following result; see Theorems \ref{3.3} and \ref{3.12}.

\begin{theorem*}
Let $f:R \rightarrow S$ be a local homomorphism of noetherian local rings, and $M$ a non-zero finitely generated or artinian $S$-module with $\id_{R}(M)\leq \edim(R)- \edim(S)$. Then the following assertions hold:
\begin{enumerate}
\item[(i)] $M$ is an injective $S$-module, i.e. $M\cong E_{S}(l)^{\type_{S}(M)}$ where $l$ is the residue field of $S$.
\item[(ii)] $f$ is an exceptional complete intersection map.
\item[(iii)] $\id_{R}(M)= \edim(R)- \edim(S)= \fd_{R}(S)$.
\end{enumerate}
In particular, if $M$ is injective as an $R$-module, then $\edim(R)= \edim(S)$ and $f$ is flat.
\end{theorem*}

We further prove an analogue of \cite[Theorem 3.1]{BIK} for artinian modules; see Corollary \ref{3.13}. Along the way, we collect a few lemmas with suitable references for the convenience of the reader, and also establish a handful of lemmas which might be of independent interest.

\section{Basic Definitions and Observations}

In what follow, all rings are assumed to be commutative with unity. Given a local ring $(R,\mathfrak{m},k)$, $\mathfrak{m}$ denotes its unique maximal ideal, $k\cong R/ \mathfrak{m}$ is its residue field, and $\widehat{R}$ designates its $\mathfrak{m}$-adic completion. A ring homomorphism $f:(R,\mathfrak{m}) \rightarrow (S,\mathfrak{n})$ is local if $f(\mathfrak{m})\subseteq \mathfrak{n}$. Given an $R$-module $M$, the notations $\pd_{R}(M)$, $\id_{R}(M)$, and $\fd_{R}(M)$ are used for the projective, injective, and flat dimensions of $M$, respectively. Moreover, $\dim(R)$ and $\edim(R)$ indicate the Krull and the embedding dimensions of $R$, respectively.

We begin with a remark on Andr\'{e}-Quillen homology.

\begin{remark} \label{2.1}
Let $f:R\rightarrow S$ be a ring homomorphism. By developing the machinery of non-abelian homological algebra in the framework of Quillen's model category theory, one can construct the cotangent complex $\mathbb{L}^{S|R}$ associated to $f$ as a well-defined object in the derived category $\mathcal{D}(S)$. Given any $S$-module $M$, the cotangent complex $\mathbb{L}^{S|R}$ is used to define the Andr\'{e}-Quillen homology modules of $S$ over $R$ with coefficients in $M$ as
$$D_{i}(S|R;M):= \Tor_{i}^{S}\left(\mathbb{L}^{S|R},M\right)$$
for every $i\geq 0$. For any $i\geq 0$, the Andr\'{e}-Quillen homology $D_{i}(S|R;M)$ is functorial in all three arguments. One further has
$$D_{0}(S|R;M) \cong \Omega_{S|R}\otimes_{S}M$$
where $\Omega_{S|R}$ is the module of K\"{a}hler differentials of $S$ over $R$. If $f$ is surjective with $\Ker(f)=\mathfrak{a}$, then $D_{0}(S|R;M)=0$ and $$D_{1}(S|R;M) \cong (\mathfrak{a}/ \mathfrak{a}^{2})\otimes_{S}M.$$

Given ring homomorphisms $R\rightarrow S\rightarrow T$ and a $T$-module $M$, one has the following Jacobi-Zariski exact sequence of $T$-modules:
\small
\begin{gather*}
\cdots \rightarrow D_{1}\left(S|R;M \right)\rightarrow D_{1}\left(T|R;M \right) \rightarrow D_{1}\left(T|S;M \right)\rightarrow D_{0}\left(S|R;M \right)\rightarrow D_{0}\left(T|R;M \right)\rightarrow D_{0}\left(T|S;M \right) \rightarrow 0
\end{gather*}
\normalsize

If $f:(R,\mathfrak{m},k) \rightarrow S$ is a local epimorphism of noetherian local rings with $\Ker(f)=\mathfrak{a}$, then
$$D_{1}(S|R;k) \cong (\mathfrak{a}/ \mathfrak{a}^{2})\otimes_{S}k \cong \mathfrak{a} / \mathfrak{m}\mathfrak{a},$$
so the minimum number of generators of $\mathfrak{a}$ can be computed as
$$\mu(\mathfrak{a})= \rank_{k}(\mathfrak{a}/ \mathfrak{m}\mathfrak{a}) = \rank_{k}\left(D_{1}(S|R;k)\right),$$
which in particular yields
$$\edim(R)=\mu(\mathfrak{m})= \rank_{k}\left(D_{1}(k|R;k)\right).$$

It is worth noting that the general theory of Andr\'{e}-Quillen homology uses simplicial resolutions; see \cite{An}, \cite{GS}, \cite{Iy}, \cite{Qu1}, \cite{Qu2}, and \cite{Qu3}. However, the lower Andr\'{e}-Quillen homologies $D_{i}(S|R;M)$ for $i=0,1,2$ can be constructed without any appeal to simplicial resolutions; see \cite{LS}.
\end{remark}

Using Andr\'{e}-Quillen homology, one can prove the following lemma.

\begin{lemma} \label{2.2}
Let $f:(R,\mathfrak{m}) \rightarrow S$ be a flat local homomorphism of noetherian local rings whose closed fiber $S/ \mathfrak{m}S$ is regular. Then we have:
$$\edim\left(S/ \mathfrak{m}S\right)= \edim(S)- \edim(R)= \dim(S)- \dim(R)$$
\end{lemma}

\begin{proof}
See \cite[2.6]{BIK}.
\end{proof}

We next recall the notion of an exceptional complete intersection map.

\begin{remark} \label{2.3}
Let $f:(R,\mathfrak{m},k) \rightarrow S$ be a local homomorphism of noetherian local rings. Then $f$ admits a Cohen factorization, i.e. it fits into a commutative diagram
\begin{equation*}
  \begin{tikzcd}
  R \arrow{r}{f} \arrow{d}[swap]{\dot{f}} & S \arrow{d}
  \\
  R' \arrow{r}{f'} & \widehat{S}
\end{tikzcd}
\end{equation*}

\noindent
of noetherian local rings in which $\dot{f}$ is flat with regular closed fiber $R'/ \mathfrak{m}R'$ and $f'$ is surjective; see \cite[Theorem 1.1]{AFH}. Accordingly, $f$ is said to be a \textit{complete intersection map} if there is a Cohen factorization of $f$ such that $\Ker(f')$ is generated by a regular sequence on $R'$. This property is independent of the choice of the Cohen factorization; see \cite[Theorem 1.2]{Av}. In particular, $f$ is complete intersection if and only if $f'$ is complete intersection. Moreover, the class of complete intersection maps is closed under composition and flat base change. When $f$ is complete intersection, there is an inequality
$$\edim(R)-\dim(R) \leq \edim(S)-\dim(S);$$
see \cite[Lemma 2.13]{BIK}. Then $f$ is said to be an \textit{exceptional complete intersection map} if equality holds above. When $f$ is surjective, this property is equivalent to the condition that $\Ker(f)$ can be generated by a regular sequence whose image in $\mathfrak{m}/ \mathfrak{m}^{2}$ is a linearly independent set over $k$; see \cite[Lemma 3.2]{ILP}. It follows from Lemma \ref{2.2} that $f$ is exceptional complete intersection if and only if $f'$ is exceptional complete intersection. For more information on (exceptional) complete intersection maps, refer to \cite{Av}, \cite{BIK}, \cite{ILP}, and \cite{BILP}.
\end{remark}

The following lemma will be used later.

\begin{lemma} \label{2.4}
Let $f:R \rightarrow S$ be a surjective exceptional complete intersection map, $k$ the residue field of $S$, and $\pi:S\rightarrow k$ the quotient map. Then $D_{1}(\pi|R;k): D_{1}(S|R;k) \rightarrow D_{1}(k|R;k)$ is injective.
\end{lemma}

\begin{proof}
Follows from the proof of \cite[Lemma 2.16]{BIK}.
\end{proof}

\section{Main Results}

In this section, we prove our main result, i.e. the theorem outlined in the introduction. We first handle the artinian case.

\begin{lemma} \label{3.1}
Let $(R,\mathfrak{m},k)$ be a noetherian local ring, and $M$ a non-zero finitely generated or artinian $R$-module. Let $E_{R}(k)$ denote the injective envelope of $k$, and $\type_{R}(M)= \rank_{k}\left(\Ext_{R}^{\depth_{R}(\mathfrak{m},M)}(k,M)\right)$. Then $M$ is an injective $R$-module if and only if $M\cong E_{R}(k)^{\type_{R}(M)}$.
\end{lemma}

\begin{proof}
Suppose that $M$ is injective. If $M$ is finitely generated, then $R$ is artinian; see \cite[Exercise 3.1.23]{BH}. Thus in any case, $M$ is artinian, so $\Ass_{R}(M)=\{\mathfrak{m}\}$, whence $\depth_{R}(\mathfrak{m},M)=0$. But then the Bass number $\mu_{0}^{R}(\mathfrak{m},M)= \rank_{k}\left(\Ext_{R}^{0}(k,M)\right)= \type_{R}(M)$, so $M\cong E_{R}(M)\cong E_{R}(k)^{\type_{R}(M)}$. The converse is clear.
\end{proof}

\begin{lemma} \label{3.2}
Let $f:R\rightarrow S$ be a homomorphism of noetherian rings, $X$ an $S$-complex that is considered an $R$-complex via $f$, and $X^{\vee}= \Hom_{S}\left(X,E_{S}(l)\right)$ where $l$ is the residue field of $S$. Then the following assertions hold:
\begin{enumerate}
\item[(i)] $\id_{R}(X)=\fd_{R}(X^{\vee})$ provided $H_{i}(X)=0$ for $|i| \gg 0$.
\item[(ii)] $\fd_{R}(X)=\id_{R}(X^{\vee})$.
\end{enumerate}
\end{lemma}

\begin{proof}
(i): Let $\mathfrak{a}$ be an ideal of $R$, and $L$ a free resolution of $R/ \mathfrak{a}$ in which $L_{i}$ is finitely generated for every $i\geq 0$. Then using the Hom Evaluation (see \cite[Lemma 4.4 (I)]{AF}), we get
\begin{equation*}
\begin{split}
 \Ext_{R}^{i}(R/ \mathfrak{a},X)^{\vee} & = \Hom_{S}\left(\Ext_{R}^{i}(R/ \mathfrak{a},X),E_{S}(l)\right) \\
 & = \Hom_{S}\left(H_{-i}\left(\Hom_{R}(L,X)\right),E_{S}(l)\right) \\
 & \cong H_{i}\left(\Hom_{S}\left(\Hom_{R}(L,X),E_{S}(l)\right)\right) \\
 & \cong H_{i}\left(L\otimes_{R} \Hom_{S}\left(X,E_{S}(l)\right)\right) \\
 & = \Tor_{i}^{R}\left(R/ \mathfrak{a},\Hom_{S}\left(X,E_{S}(l)\right)\right) \\
 & = \Tor_{i}^{R}\left(R/ \mathfrak{a},X^{\vee}\right)
\end{split}
\end{equation*}
for every $i\in \mathbb{Z}$. Since $E_{S}(l)$ is faithfully injective, the isomorphism implies that for any $i\in \mathbb{Z}$, we have $\Ext_{R}^{i}(R/ \mathfrak{a},X)=0$ if and only if $\Tor_{i}^{R}\left(R/ \mathfrak{a},X^{\vee}\right)=0$. Now the result follows from \cite[Corollaries 2.5.I and 2.5.F]{AF}.

(ii): Let $\mathfrak{a}$ be an ideal of $R$, and $L$ a free resolution of $R/ \mathfrak{a}$. Then using the Hom-Tensor Adjunction, see \cite[A.2.8]{Ch}, we get
\begin{equation*}
\begin{split}
 \Tor_{i}^{R}(R/ \mathfrak{a},X)^{\vee} & = \Hom_{S}\left(\Tor_{i}^{R}(R/ \mathfrak{a},X),E_{S}(l)\right) \\
 & = \Hom_{S}\left(H_{i}\left(L\otimes_{R}X\right),E_{S}(l)\right) \\
 & \cong H_{-i}\left(\Hom_{S}\left(L\otimes_{R}X,E_{S}(l)\right)\right) \\
 & \cong H_{-i}\left(\Hom_{R}\left(L,\Hom_{S}\left(X,E_{S}(l)\right)\right)\right) \\
 & = \Ext_{R}^{i}\left(R/ \mathfrak{a},\Hom_{S}\left(X,E_{S}(l)\right)\right) \\
 & = \Ext_{R}^{i}\left(R/ \mathfrak{a},X^{\vee}\right)
\end{split}
\end{equation*}
for every $i\in \mathbb{Z}$. As in (i), the isomorphism implies that for any $i\in \mathbb{Z}$, we have $\Tor_{i}^{R}(R/ \mathfrak{a},X)=0$ if and only if $\Ext_{R}^{i}\left(R/ \mathfrak{a},X^{\vee}\right)=0$. Now the result follows from \cite[Corollaries 2.5.I and 2.5.F]{AF}.
\end{proof}

\begin{theorem} \label{3.3}
Let $f:(R,\mathfrak{m},k) \rightarrow (S,\mathfrak{n},l)$ be a local homomorphism of noetherian local rings, and $M$ a non-zero artinian $S$-module with $\id_{R}(M)\leq \edim(R)- \edim(S)$. Then the following assertions hold:
\begin{enumerate}
\item[(i)] $M$ is an injective $S$-module, i.e. $M\cong E_{S}(l)^{\type_{S}(M)}$.
\item[(ii)] $f$ is an exceptional complete intersection map.
\item[(iii)] $\id_{R}(M)= \edim(R)- \edim(S) = \fd_{R}(S)$.
\end{enumerate}
In particular, if $M$ is injective as an $R$-module, then $\edim(R)= \edim(S)$ and $f$ is flat.
\end{theorem}

\begin{proof}
If $M$ is an injective $S$-module, then by Lemma \ref{3.1}, $M \cong E_{S}(l)^{\type_{S}(M)}$, so Lemma \ref{3.2} (ii) implies that
$$\id_{R}(M)=\id_{R}\left(E_{S}(l)\right)=\id_{R}(S^{\vee})= \fd_{R}(S).$$
That is to say, the second equality in (iii) follows from (i).

We next reduce to the case where $S$ is $\mathfrak{n}$-adically complete. As $M$ is an artinian $S$-module, hence $\mathfrak{n}$-torsion, it has an $\widehat{S}$-module structure that restricts to its original $S$-module structure via the completion map $S\rightarrow \widehat{S}$, and $M\cong M\otimes_{S}\widehat{S}$ both as $S$-modules and $\widehat{S}$-modules; see \cite[Proposition 2.1.15 and Corollary 2.2.6]{SS}. Also, $M$ is an artinian $\widehat{S}$-module. We further note that $\edim\left(\widehat{S}\right)=\edim(S)$,
$$E_{\widehat{S}}\left(\widehat{S}/ \mathfrak{n}\widehat{S}\right) \cong E_{S}(S/ \mathfrak{n}) \cong E_{S}(l),$$
and
$$\type_{\widehat{S}}(M)= \type_{\widehat{S}}\left(M\otimes_{S}\widehat{S}\right)= \type_{S}(M).$$
By Lemma \ref{3.1}, $M$ is an injective $S$-module if and only if $M\cong E_{S}(l)^{\type_{S}(M)}$. Similarly, $M$ is an injective $\widehat{S}$-module if and only if $M\cong E_{\widehat{S}}\left(\widehat{S}/ \mathfrak{n}\widehat{S}\right)^{\type_{\widehat{S}}(M)}$. It follows that $M$ is an injective $S$-module if and only if $M$ is an injective $\widehat{S}$-module. Therefore, we can replace $S$ with $\widehat{S}$ and assume that $S$ is $\mathfrak{n}$-adically complete. Consequently, the Matlis duality theory implies that $M^{\vee}=\Hom_{S}\left(M,E_{S}(l)\right)$ is a finitely generated $S$-module.

Now by Lemma \ref{3.2} (i), we have
$$\fd_{R}(M^{\vee})=\id_{R}(M)\leq \edim(R)- \edim(S).$$
Therefore, \cite[Theorem 3.1]{BIK} implies that $M^{\vee}$ is a free $S$-module, $f$ is an exceptional complete intersection map, and
$$\id_{R}(M)= \fd_{R}(M^{\vee})= \edim(R)- \edim(S).$$
Furthermore, $\id_{S}(M)= \fd_{S}(M^{\vee})=0$, so $M$ is an injective $S$-module.
\end{proof}

We next need a series of lemmas before we can establish the finitely generated case of the theorem outlined in the introduction.

\begin{lemma} \label{3.4}
Let $f:(R,\mathfrak{m}) \rightarrow S$ be a flat local homomorphism of noetherian local rings, and $M$ a non-zero finitely generated $R$-module. Then we have:
$$\id_{S}(M\otimes_{R}S) = \id_{R}(M) + \id_{S/ \mathfrak{m}S}(S/ \mathfrak{m}S)$$
\end{lemma}

\begin{proof}
See \cite[Corollary 1]{FT}.
\end{proof}

\begin{corollary} \label{3.5}
Let $(R,\mathfrak{m},k)$ be a noetherian local ring, and $M$ a finitely generated $R$-module. Then $\id_{\widehat{R}}\left(M\otimes_{R}\widehat{R}\right) = \id_{R}(M)$. In particular, $M$ is an injective $R$-module if and only if $M\otimes_{R}\widehat{R}$ is an injective $\widehat{R}$-module.
\end{corollary}

\begin{proof}
Since the completion map $R\rightarrow \widehat{R}$ is flat and local, Lemma \ref{3.4} yields
$$\id_{\widehat{R}}\left(M\otimes_{R}\widehat{R}\right) = \id_{R}(M) + \id_{\widehat{R}/ \mathfrak{m}\widehat{R}}\left(\widehat{R}/ \mathfrak{m}\widehat{R}\right).$$
However, $\widehat{R}/ \mathfrak{m}\widehat{R} \cong R/ \mathfrak{m} \cong k$, so $\id_{\widehat{R}/ \mathfrak{m}\widehat{R}}\left(\widehat{R}/ \mathfrak{m}\widehat{R}\right)= \id_{k}(k)=0$ as any vector space is injective. Thus the result follows.
\end{proof}

\begin{lemma} \label{3.6}
Let $R$ be a noetherian local ring, and $M$ a finitely generated $R$-module with $\id_{R}(M)< \infty$. Then $\ann_{R}(M)$ is either zero or contains a non-zerodivisor of $R$.
\end{lemma}

\begin{proof}
See \cite[Theorem 4.1]{LV}.
\end{proof}

\begin{lemma} \label{3.7}
Let $(R,\mathfrak{m})$ be a noetherian local ring, $a\in \mathfrak{m}\backslash \mathfrak{m}^{2}$ a non-zerodivisor of $R$, and $M$ a finitely generated $R/ \langle a \rangle$-module. Then the following assertions hold:
\begin{enumerate}
\item[(i)] $\pd_{R/ \langle a \rangle}(M)=\pd_{R}(M)-1$.
\item[(ii)] $\id_{R/ \langle a \rangle}(M)=\id_{R}(M)-1$.
\end{enumerate}
\end{lemma}

\begin{proof}
See \cite[Corollary 27.4]{Na} and \cite[Theorem 3.1]{LV}.
\end{proof}

\begin{lemma} \label{3.8}
Let $R$ be a noetherian ring, $M$ a finitely generated $R$-module, and $\underline{a}=a_{1},...,a_{n}\in R$ a regular sequence on $M$. If $L$ is a free resolution of $M$, then $K^{R}(\underline{a})\otimes_{R} L$ is a free resolution of $M/ \langle \underline{a} \rangle M$.
\end{lemma}

\begin{proof}
Since $L$ is a free resolution of $M$, there is a quasi-isomorphism $L \rightarrow M$. As the Koszul complex $K^{R}(\underline{a})$ is a bounded complex of free $R$-modules, we get a quasi-isomorphism $K^{R}(\underline{a}) \otimes_{R}L \rightarrow K^{R}(\underline{a}) \otimes_{R}M$. Since $\underline{a}=a_{1},...,a_{n}$ is a regular sequence on $M$, \cite[Corollary 1.6.14 (a)]{BH} implies that
$$H_{i}\left(K^{R}(\underline{a}) \otimes_{R}M\right) = H_{i}(\underline{a};M) \cong
\begin{cases}
M/ \langle \underline{a} \rangle M & \text{if } i=0 \\
0 & \text{if } i > 0,
\end{cases}$$
so there is a quasi-isomorphism $K^{R}(\underline{a}) \otimes_{R}M \rightarrow M/ \langle \underline{a} \rangle M$. As a result, we get a quasi-isomorphism $K^{R}(\underline{a}) \otimes_{R}L \rightarrow M/ \langle \underline{a} \rangle M$, so $K^{R}(\underline{a})\otimes_{R} L$ is a free resolution of $M/ \langle \underline{a} \rangle M$.
\end{proof}

\begin{remark} \label{3.9}
Let $R$ and $S$ be two rings, $M$ a symmetric $R$-$R$-bimodule, i.e. $ax=xa$ for every $a\in R$ and $x\in M$, and $N$ and $K$ two $R$-$S$-bimodules. Considering $M$ as a left $R$-module, $N$ as a right $S$-module, and $K$ as an $R$-$S$-bimodule, we have the Hom Swap Isomorphism
$$\Hom_{R}\left(M,\Hom_{S}(N,K)\right)\cong \Hom_{S}\left(N,\Hom_{R}(M,K)\right);$$
see \cite[A.2.9]{Ch}. On the other hand, considering $M$ as a right $R$-module, $N$ as an $R$-$S$-bimodule, and $K$ as a right $S$-module, we have the Hom-Tensor Adjoint Isomorphism
$$\Hom_{R}\left(M,\Hom_{S}(N,K)\right)\cong \Hom_{S}(M\otimes_{R}N,K);$$
see \cite[A.2.8]{Ch}. One might be tempted to merge the two isomorphisms as their left hand sides seem to coincide. However, in the first isomorphism, $M$ is a left $R$-module and $\Hom_{S}(N,K)$ is a left $R$-module via the left $R$-module structure of $K$. Similarly, in the second isomorphism, $M$ is a right $R$-module and $\Hom_{S}(N,K)$ is a right $R$-module via the left $R$-module structure of $N$. As a result, if this $R$-$R$-bimodule structure of $\Hom_{S}(N,K)$ is symmetric, then we conclude that the left hand sides of the above isomorphisms coincide, so we get:
$$\Hom_{S}(M\otimes_{R}N,K)\cong \Hom_{S}\left(N,\Hom_{R}(M,K)\right)$$
But if the $R$-$R$-bimodule structure of $\Hom_{S}(N,K)$ is not symmetric, then the above isomorphism need not hold.

For example, let $K$ be a field. Considering the polynomial ring $K[X]$ as a $K[X]$-$K$-bimodule in a natural way, we notice that the $K[X]$-$K[X]$-bimodule $\Hom_{K}\left(K[X],K[X]\right)$ is not symmetric. Indeed, the left $K[X]$-module structure of $\Hom_{K}\left(K[X],K[X]\right)$ comes from the left $K[X]$-module structure of $K[X]$ in the second argument, and the right $K[X]$-module structure of $\Hom_{K}\left(K[X],K[X]\right)$ comes from the left $K[X]$-module structure of $K[X]$ in the first argument. Consider the elements $X\in K[X]$ and $f\in \Hom_{K}\left(K[X],K[X]\right)$ defined by $f(1)=1$ and $f(X^{n})=X^{n+1}$ for every $n\geq 1$. Then we have
$$(Xf)(1)= Xf(1)= X \neq X^{2} = f(X) = f(X1) = (fX)(1),$$
so $Xf \neq fX$. We further observe that:
\begin{equation*}
\begin{split}
 \Hom_{K}\left(K[X]\otimes_{K[X]}K[X]/ \langle X \rangle,K[X]\right) & \cong \Hom_{K}\left(K[X]/ \langle X \rangle,K[X]\right) \\
 & \cong \Hom_{K}\left(K,K[X]\right) \\
 & \cong K[X]
\end{split}
\end{equation*}
and
\begin{equation*}
\begin{split}
 \Hom_{K}\left(K[X],\Hom_{K[X]}\left(K[X]/ \langle X \rangle,K[X]\right)\right) & \cong \Hom_{K}\left(K[X],\left(0:_{K[X]}\langle X \rangle\right)\right) \\
 & = \Hom_{K}(K[X],0) \\
 & = 0
\end{split}
\end{equation*}
Thus the left hand sides cannot be isomorphic.
\end{remark}

\begin{lemma} \label{3.10}
Let $f:(R,\mathfrak{m},k) \rightarrow S$ be a local homomorphism of noetherian local rings, and $M$ a finitely generated $S$-module that is considered an $R$-module via $f$. Then the following assertions are equivalent for any given $n\geq 0$:
\begin{enumerate}
\item[(i)] $\id_{R}(M)\leq n$.
\item[(ii)] $\Ext_{R}^{i}(k,M)=0$ for every $i\geq n+1$.
\end{enumerate}
\end{lemma}

\begin{proof}
See \cite[Proposition 5.5 (I)]{AF}.
\end{proof}

\begin{lemma} \label{3.11}
Let $f:(R,\mathfrak{m},k) \rightarrow (S,\mathfrak{n},l)$ be a flat local homomorphism of noetherian local rings with $S/ \mathfrak{m}S$ regular, $g:S\rightarrow T$ a local homomorphism of noetherian local rings, and $M$ a finitely generated $T$-module. Then we have:
$$\id_{S}(M)\leq \id_{R}(M)+ \edim(S/ \mathfrak{m}S)$$
\end{lemma}

\begin{proof}
Since $f$ is local, we have $\mathfrak{m}S\subseteq \mathfrak{n}$. Considering the noetherian local ring $(S/ \mathfrak{m}S,\mathfrak{n}/ \mathfrak{m}S)$, we set $\mu(\mathfrak{n}/ \mathfrak{m}S) = \edim(S/ \mathfrak{m}S) = n$. Let $\underline{a}=a_{1},...,a_{n}\in \mathfrak{n}$ be such that $\{a_{1}+\mathfrak{m}S,...,a_{n}+\mathfrak{m}S\}$ is a minimal generating set for $\mathfrak{n}/ \mathfrak{m}S$. In particular, $\mathfrak{n} = \langle \underline{a} \rangle + \mathfrak{m}S$. As $S/ \mathfrak{m}S$ is regular, we conclude that $a_{1}+\mathfrak{m}S,...,a_{n}+\mathfrak{m}S$ is a regular sequence on the ring $S/ \mathfrak{m}S$, so $\underline{a}=a_{1},...,a_{n}$ is a regular sequence on the $S$-module $S/\mathfrak{m}S$. Set the Koszul complex $K=K^{S}(\underline{a})$. Moreover, let $L$ be a free resolution for the $R$-module $k$. Since $f$ is flat, we infer that $S\otimes_{R}L$ is a free resolution for the $S$-module $S\otimes_{R}k \cong S/ \mathfrak{m}S$. It follows from Lemma \ref{3.8} that $K\otimes_{R}L \cong K\otimes_{S}(S\otimes_{R}L)$ is a free resolution for the $S$-module
$$(S/ \mathfrak{m}S) / \langle \underline{a} \rangle (S/ \mathfrak{m}S) \cong S/ (\langle \underline{a} \rangle + \mathfrak{m}S) = S/ \mathfrak{n} \cong l.$$
Consider the third quadrant $S$-bicomplex $W$ with $W_{p,q} = \Hom_{S}\left(K_{-p},\Hom_{R}(L_{-q},M)\right)$ for every $p,q\in \mathbb{Z}$ and the differentials induced by those of $K$ and $L$. It can be seen by inspection that $\Tot(W)=\Hom_{S}\left(K,\Hom_{R}(L,M)\right)$. We compute the first spectral sequence arising from $W$. For any $p\leq 0$, $K_{-p}$ is a free $S$-module, so the functor $\Hom_{S}\left(K_{-p},-\right)$ is exact. Thus we get
\begin{equation*}
\begin{split}
 E^{1}_{p,q} & = H_{p,q}^{v}(W) \\
 & = H_{q}\left(\Hom_{S}\left(K_{-p},\Hom_{R}(L,M)\right)\right) \\
 & \cong \Hom_{S}\left(K_{-p},H_{q}\left(\Hom_{R}(L,M)\right)\right) \\
 & = \Hom_{S}\left(K_{-p},\Ext_{R}^{-q}(k,M)\right)
\end{split}
\end{equation*}
for every $p,q \leq 0$. Next, we obtain
\begin{equation*}
\begin{split}
 E^{2}_{p,q} & = H_{p,q}^{h}\left(H^{v}(W)\right) \\
 & = H_{p}\left(\Hom_{S}\left(K,\Ext_{R}^{-q}(k,M)\right)\right) \\
 & \cong H_{p+n}\left(\underline{a};\Ext_{R}^{-q}(k,M)\right)
\end{split}
\end{equation*}
for every $p,q\leq 0$. We note that $L$ is a complex of symmetric $R$-$R$-bimodules. Moreover, by restriction of scalars via $f$, we notice that $K$ is a complex of $R$-$S$-bimodules and $M$ is an $R$-$S$-bimodule. Accordingly, for any $i\geq 0$, we deduce that $\Hom_{S}(K_{i},M)$ is a left $R$-module via the left $R$-module structure of $M$, and it is a right $R$-module via the left $R$-module structure of $K_i$. If $\varphi_{i}\in \Hom_{S}(K_{i},M)$, $x\in K_{i}$, and $r\in R$, then we have:
$$(r\varphi_{i})(x)=r\varphi_{i}(x)=\varphi_{i}(x)f(r)=\varphi_{i}\left(xf(r)\right)=\varphi_{i}(rx)=(\varphi_{i}r)(x)$$
Hence $r \varphi_{i}= \varphi_{i} r$. This shows that $\Hom_{S}(K_{i},M)$ is a symmetric $R$-$R$-bimodule for every $i\geq 0$. Therefore, by Remark \ref{3.9}, the following isomorphism of $S$-complexes holds:
$$\Tot(W)=\Hom_{S}\left(K,\Hom_{R}(L,M)\right)\cong \Hom_{S}(L\otimes_{R}K,M)$$
Bearing in mind that $L\otimes_{R}K \cong K\otimes_{R}L$ is a free resolution for the $S$-module $l$, we get
$$H_{i}\left(\Tot(W)\right) \cong H_{i}\left(\Hom_{S}(L\otimes_{R}K,M)\right) = \Ext_{S}^{-i}(l,M)$$
for every $i\in \mathbb{Z}$. Therefore, we wind up with the following third quadrant spectral sequence:
$$E^{2}_{p,q}=H_{p+n}\left(\underline{a};\Ext_{R}^{-q}(k,M)\right) \Rightarrow \Ext_{S}^{-p-q}(l,M)$$
Now if $p+q<-\id_{R}(M)-n$, then either $p<-n$ or $q<-\id_{R}(M)$. If $p<-n$, then $E^{2}_{p,q}=0$, and if $q<-\id_{R}(M)$, then $\Ext^{-q}_{R}(k,M)=0$, so $E^{2}_{p,q}=0$. It follows that $\Ext^{i}_{S}(l,M)=0$ for every $i>\id_{R}(M)+n$, so by Lemma \ref{3.10}, $\id_{S}(M)\leq \id_{R}(M)+n$.
\end{proof}

We are now ready to deal with the finitely generated case.

\begin{theorem} \label{3.12}
Let $f:(R,\mathfrak{m},k) \rightarrow (S,\mathfrak{n},l)$ be a local homomorphism of noetherian local rings, and $M$ a non-zero finitely generated $S$-module with $\id_{R}(M)\leq \edim(R)- \edim(S)$. Then the following assertions hold:
\begin{enumerate}
\item[(i)] $M$ is an injective $S$-module, i.e. $M\cong E_{S}(l)^{\type_{S}(M)}$.
\item[(ii)] $f$ is an exceptional complete intersection map.
\item[(iii)] $\id_{R}(M) = \edim(R)- \edim(S) = \fd_{R}(S)$.
\end{enumerate}
In particular, if $M$ is injective as an $R$-module, then $\edim(R)= \edim(S)$ and $f$ is flat.
\end{theorem}

\begin{proof}
If $M$ is an injective $S$-module, then by Lemma \ref{3.1}, $M \cong E_{S}(l)^{\type_{S}(M)}$, so Lemma \ref{3.2} (ii) implies that $$\id_{R}(M)=\id_{R}\left(E_{S}(l)\right)=\id_{R}(S^{\vee})= \fd_{R}(S).$$
That is to say, the second equality in (iii) follows from (i).

We first reduce to the case where $S$ is $\mathfrak{n}$-adically complete. We note that $M\otimes_{S}\widehat{S}$ is a non-zero finitely generated $\widehat{S}$-module. Let $\mathfrak{a}$ be an ideal of $R$, and $L$ a free resolution of $R/ \mathfrak{a}$ in which $L_{i}$ is finitely generated for every $i\geq 0$. Then using the Tensor Evaluation (see \cite[Page 12]{Ch}), we get
\begin{equation*}
\begin{split}
 \Ext_{R}^{i}(R/ \mathfrak{a},M)\otimes_{S}\widehat{S} & = H_{-i}\left(\Hom_{R}(L,M)\right)\otimes_{S}\widehat{S} \\
 & \cong H_{-i}\left(\Hom_{R}(L,M)\otimes_{S}\widehat{S}\right) \\
 & \cong H_{-i}\left(\Hom_{R}\left(L,M\otimes_{S}\widehat{S}\right)\right) \\
 & = \Ext_{R}^{i}\left(R/ \mathfrak{a},M\otimes_{S}\widehat{S}\right)
\end{split}
\end{equation*}
for every $i\geq 0$. Since the completion map $S\rightarrow \widehat{S}$ is faithfully flat, the above isomorphism implies that for any $i\geq 0$, we have $\Ext_{R}^{i}(R/ \mathfrak{a},M)=0$ if and only if $\Ext_{R}^{i}\left(R/ \mathfrak{a},M\otimes_{S}\widehat{S}\right)=0$. Therefore, we infer from \cite[Proposition 3.1.10]{BH} that $\id_{R}(M)=\id_{R}\left(M\otimes_{S}\widehat{S}\right)$. On the other hand, Corollary \ref{3.5} implies that $M$ is an injective $S$-module if and only if $M\otimes_{S}\widehat{S}$ is an injective $\widehat{S}$-module. All in all, we can replace $S$ with $\widehat{S}$ and $M$ with $M\otimes_{S}\widehat{S}$, and assume that $S$ is $\mathfrak{n}$-adically complete.

We next reduce to the case where $f$ is surjective. By Remark \ref{2.3}, $f$ has a Cohen factorization
$$R \xrightarrow{\dot{f}} R' \xrightarrow{f'} S$$
in which $\dot{f}$ is flat with regular closed fiber $R'/ \mathfrak{m}R'$ and $f'$ is surjective. Moreover, $f$ is exceptional complete intersection if and only if $f'$ is exceptional complete intersection. In addition, in light of Lemma \ref{3.11}, Lemma \ref{2.2}, and our hypothesis, we see that:
\begin{equation*}
\begin{split}
 \id_{R'}(M) & \leq \id_{R}(M) + \edim\left(R'/ \mathfrak{m}R'\right) \\
 & = \id_{R}(M) + \edim(R')-\edim(R) \\
 & \leq \edim(R)-\edim(S) + \edim(R')-\edim(R) \\
 & = \edim(R')-\edim(S)
\end{split}
\end{equation*}
If $\id_{R'}(M)=\edim(R')-\edim(S)$, then the above display shows that $\id_{R}(M)=\edim(R)-\edim(S)$. Hence it suffices to prove the result for $f'$. Therefore, we can assume that $f$ is surjective, so that $M$ is a finitely generated $R$-module as well.

We now consider the special case in which $M$ is a faithful $S$-module, i.e. $\ann_{S}(M)=0$. Then $\ann_{R}(M)=f^{-1}\left(\ann_{S}(M)\right)=\Ker(f)$. We argue by induction on $\edim(R)-\edim(S)$.

If $\edim(R)-\edim(S)=0$, then our assumption implies that $\id_{R}(M)=0$, i.e. $M$ is an injective $R$-module. As $M$ is a finitely generated $R$-module, we deduce from Lemma \ref{3.1} that $M \cong E_{R}(k)^{\type_{R}(M)}$. But $E_{R}(k)\cong R^{\vee}=\Hom_{R}\left(R,E_{R}(k)\right)$, so $\ann_{R}\left(E_{R}(k)\right)=\ann_{R}(R)=0$. It follows that $\Ker(f)=\ann_{R}(M)=\ann_{R}\left(E_{R}(k)\right)=0$, so $f$ is an isomorphism. As a result, $M$ is an injective $S$-module and $f$ is an exceptional complete intersection map.

Next suppose that $\edim(R)-\edim(S)\geq 1$. Then $f$ is not an isomorphism, so $\ann_{R}(M)=\Ker(f)\neq 0$. As $\id_{R}(M)< \infty$, Lemma \ref{3.6} implies that there exists a non-zerodivisor in $\Ker(f)$, so $\Ker(f)\nsubseteq \mathcal{Z}(R)=\bigcup_{\mathfrak{p}\in \Ass(R)}\mathfrak{p}$. On the other hand, $\edim(S)=\edim\left(R/ \Ker(f)\right)\leq \edim(R)$ with equality if and only if $\Ker(f)\subseteq \mathfrak{m}^{2}$. But $\edim(R)-\edim(S)\geq 1$, so $\Ker(f)\nsubseteq \mathfrak{m}^{2}$. By the Prime Avoidance Lemma (see \cite[Theorem 3.61]{Sh}), we infer that $\Ker(f)\nsubseteq \mathfrak{m}^{2}\cup \bigcup_{\mathfrak{p}\in \Ass(R)}\mathfrak{p}$, so there exists a non-zerodivisor $a\in \Ker(f)\backslash \mathfrak{m}^{2}$. Consider the following factorization of $f$:
$$R \rightarrow R/ \langle a \rangle \xrightarrow{\bar{f}} S$$
As $a\in \mathfrak{m}\backslash \mathfrak{m}^{2}$, Lemma \ref{3.7} (ii) implies that $\id_{R/ \langle a \rangle}(M)=\id_{R}(M)-1$. Also, $\edim\left(R/ \langle a \rangle\right)=\edim(R)-1$ and $\dim\left(R/ \langle a \rangle\right)=\dim(R)-1$. Apply the induction hypothesis to $\bar{f}$ to conclude that $M$ is an injective $S$-module and $\bar{f}$ is an exceptional complete intersection map. Then by the choice of $a$, it follows that $f$ is an exceptional complete intersection map as well.

We finally consider the general case. Set $\mathfrak{a}=\ann_{S}(M)$, $T=S/ \mathfrak{a}$, and consider the ring homomorphisms
$$R \xrightarrow{f} S \xrightarrow{\pi} T$$
where $\pi$ is the quotient map. Clearly, $M$ is a finitely generated faithful $T$-module with
$$\id_{R}(M)\leq \edim(R)- \edim(S)\leq \edim(R)- \edim(T).$$
By the special case in the previous paragraph, $M$ is an injective $T$-module, $\pi f$ is an exceptional complete intersection map, and $\id_{R}(M)= \edim(R)- \edim(T)$. Thus the above inequalities imply that $\edim(S)= \edim(T)$.

Our goal is to show that $\mathfrak{a}=0$, and then we are done. Recall that $f$ is assumed to be surjective, so $k$ can be considered as the residue field of $S$ as well as $T$. Consider the commutative diagram
\begin{equation*}
  \begin{tikzcd}
  R \arrow{r}{f} \ar[equal]{d} & S \arrow{r}{\pi} \ar[equal]{d} & T \arrow{d}{\rho}
  \\
  R \arrow{r}{f} \ar[equal]{d} & S \arrow{r}{\varpi} \arrow{d}{\pi} & k \ar[equal]{d}
  \\
  R \arrow{r}{\pi f} & T \arrow{r}{\rho} &  k
\end{tikzcd}
\end{equation*}

\noindent
in which $\varpi$ and $\rho$ are quotient maps. In view of Remark \ref{2.1}, this diagram provides us with the commutative diagram
\begin{equation*}
  \begin{tikzcd}[row sep=3em]
  D_{1}(S|R;k) \arrow{r}{\pi_{\ast}= D_{1}(\pi |R;k)} \ar[equal]{d} & [3.5em] D_{1}(T | R;k) \arrow{r} \arrow{d}{\rho_{\ast}= D_{1}(\rho | R;k)} & D_{1}(T | S;k) \arrow{r} \arrow{d} & D_{0}(S|R;k)=0
  \\
  D_{1}(S|R;k) \arrow{r}{\varpi_{\ast}= D_{1}(\varpi | R;k)} \arrow{d}{\pi_{\ast}= D_{1}(\pi | R;k)} & D_{1}(k | R;k) \arrow{r} \ar[equal]{d} & D_{1}(k | S;k) \arrow{r} \arrow{d} & D_{0}(S|R;k)=0
  \\
  D_{1}(T | R;k) \arrow{r}{\rho_{\ast}= D_{1}(\rho | R;k)} & D_{1}(k | R;k) \arrow{r} & D_{1}(k | T;k) \arrow{r} & D_{0}(T|R;k)=0
\end{tikzcd}
\end{equation*}

\noindent
of vector spaces over $k$ with exact rows in which the vanishings follow from the surjectivity of $f$ and $\pi$. From the exactness of the second row, we get:
\begin{equation*}
\begin{split}
 \edim(S) & = \rank_{k}\left(D_{1}(k|S;k)\right) \\
 & = \rank_{k}\left(\Coker(\varpi_{\ast})\right) \\
 & = \rank_{k}\left(D_{1}(k|R;k)\right)- \rank_{k}\left(\im (\varpi_{\ast})\right) \\
 & = \edim(R)- \rank_{k}\left(\im (\varpi_{\ast})\right)
\end{split}
\end{equation*}
Similarly, from the exactness of the third row, we obtain:
\begin{equation*}
\begin{split}
 \edim(T) & = \rank_{k}\left(D_{1}(k|T;k)\right) \\
 & = \rank_{k}\left(\Coker(\rho_{\ast})\right) \\
 & = \rank_{k}\left(D_{1}(k|R;k)\right)- \rank_{k}\left(\im (\rho_{\ast})\right) \\
 & = \edim(R)- \rank_{k}\left(\im (\rho_{\ast})\right)
\end{split}
\end{equation*}
These relations conspire to yield:
$$\rank_{k}\left(\im(\rho_{\ast}\pi_{\ast})\right)= \rank_{k}\left(\im(\varpi_{\ast})\right)= \rank_{k}\left(\im(\rho_{\ast})\right)$$
But $\im(\rho_{\ast}\pi_{\ast}) \subseteq \im(\rho_{\ast})$, so we conclude that $\im(\rho_{\ast}\pi_{\ast}) = \im(\rho_{\ast})$. As $\pi f$ is a surjective exceptional complete intersection map, Lemma \ref{2.4} implies that $\rho_{\ast}$ is injective, so we deduce that $\pi_{\ast}$ is surjective. Thus from the exactness of the first row, we achieve $D_{1}(T|S;k)=0$, so $\mu(\mathfrak{a})= \rank_{k}\left(D_{1}(T|S;k)\right)=0$, whence $\mathfrak{a}=0$. This completes the proof.
\end{proof}

The following corollary is an analogue of \cite[Theorem 3.1]{BIK} for artinian modules.

\begin{corollary} \label{3.13}
Let $f:(R,\mathfrak{m},k) \rightarrow (S,\mathfrak{n},l)$ be a local homomorphism of noetherian local rings, and $M$ a non-zero artinian $S$-module with $\fd_{R}(M)\leq \edim(R)- \edim(S)$. Then the following assertions hold:
\begin{enumerate}
\item[(i)] $M$ is a free $S$-module, i.e. $M\cong S^{\mu_{S}(M)}$.
\item[(ii)] $f$ is an exceptional complete intersection map.
\item[(iii)] $\fd_{R}(M)= \edim(R)- \edim(S)= \fd_{R}(S)$.
\end{enumerate}
In particular, if $M$ is flat as an $R$-module, then $\edim(R)= \edim(S)$ and $f$ is flat.
\end{corollary}

\begin{proof}
As we have observed before, since $M$ is artinian, it has an $\widehat{S}$-module structure that restricts to its $S$-module structure via the completion map $S\rightarrow \widehat{S}$. Since the completion map is faithfully flat, we have $\fd_{\widehat{S}}(M)=\fd_{\widehat{S}}\left(M\otimes_{S}\widehat{S}\right)=\fd_{S}(M)$, so $M$ is a flat $S$-module if and only $M$ is a flat $\widehat{S}$-module. By \cite[Theorem 18.6 (iii)]{Ma}, $E_{S}(l)$ is an $\widehat{S}$-module, so we have:
\begin{equation*}
\begin{split}
 M^{\vee} & = \Hom_{S}\left(M,E_{S}(l)\right) \\
 & \cong \Hom_{S}\left(M,\Hom_{\widehat{S}}\left(\widehat{S},E_{S}(l)\right)\right) \\
 & \cong \Hom_{\widehat{S}}\left(M\otimes_{S}\widehat{S},E_{S}(l)\right) \\
 & \cong \Hom_{\widehat{S}}\left(M,E_{\widehat{S}}\left(\widehat{S}/ \mathfrak{n}\widehat{S}\right)\right)
\end{split}
\end{equation*}
Therefore, if $M$ happens to be a flat $S$-module, then $M$ is a flat $\widehat{S}$-module, so by Lemma \ref{3.2} (ii) and the above isomorphism, $M^{\vee}$ is an injective $\widehat{S}$-module. On the other hand, by the Matlis duality theory, $M^{\vee}$ is a finitely generated $\widehat{S}$-module, so by Lemma \ref{3.1}, we get:
$$M^{\vee}\cong E_{\widehat{S}}\left(\widehat{S}/ \mathfrak{n}\widehat{S}\right)^{\type_{\widehat{S}}(M^{\vee})} \cong E_{S}(l)^{\type_{\widehat{S}}(M^{\vee})}$$
Moreover, by \cite[Exercise 3.1.23]{BH}, $\widehat{S}$ is artinian, so $S$ is artinian and we have
$$\ell_{S}(M)= \ell_{\widehat{S}}\left(M\otimes_{S}\widehat{S}\right)= \ell_{\widehat{S}}(M)= \ell_{\widehat{S}}(M^{\vee})< \infty,$$
which shows that $\mu_{S}(M)$ is well-defined. Now \cite[Proposition 3.2.12 (d)]{BH} implies that
$$\type_{\widehat{S}}(M^{\vee}) = \mu_{\widehat{S}}(M)= \mu_{\widehat{S}}\left(M\otimes_{S}\widehat{S}\right) = \mu_{S}(M),$$
so we get $M^{\vee}\cong E_{S}(l)^{\mu_{S}(M)}$. As $M$ is Matlis reflexive, we obtain:
$$M\cong M^{\vee\vee}\cong \left(E_{S}(l)^{\mu_{S}(M)}\right)^{\vee} \cong \left(E_{S}(l)^{\vee}\right)^{\mu_{S}(M)} \cong \widehat{S}^{\mu_{S}(M)} \cong S^{\mu_{S}(M)}$$
In particular, $\fd_{R}(M)= \fd_{R}(S)$. That is, the second equality in (iii) follows from (i).

As in the proof of Theorem \ref{3.3}, we can reduce to the case where $S$ is $\mathfrak{n}$-adically complete, so the Matlis duality theory implies that $M^{\vee}=\Hom_{S}\left(M,E_{S}(l)\right)$ is a finitely generated $S$-module. Now by Lemma \ref{3.2} (ii), we have
$$\id_{R}(M^{\vee})=\fd_{R}(M)\leq \edim(R)- \edim(S).$$
Therefore, Theorem \ref{3.12} implies that $M^{\vee}$ is an injective $S$-module, $f$ is an exceptional complete intersection map, and
$$\fd_{R}(M)= \id_{R}(M^{\vee})= \edim(R)- \edim(S).$$
Furthermore, $\fd_{S}(M)=\id_{S}(M^{\vee})=0$, so $M$ is a flat $S$-module.
\end{proof}

We need the following lemma in order to derive another corollary from our theorem.

\begin{lemma} \label{3.14}
Let $R \rightarrow S \rightarrow T$ be local homomorphisms of noetherian local rings, and $X$ a non-exact $T$-complex such that $H_{i}(X)=0$ for $|i|\gg 0$ and $H_{i}(X)$ is a finitely generated $T$-module for every $i\in \mathbb{Z}$. If $\id_{R}(X)< \infty$ and $\id_{S}(X)< \infty$, then $\fd_{R}(S)< \infty$.
\end{lemma}

\begin{proof}
Let $k$ denote the residue field of $T$. For any $i\in \mathbb{Z}$, $H_{-i}(X)$ is a finitely generated $T$-module, so the Matlis duality theory implies that
$$H_{i}(X^{\vee})= H_{i}\left(\Hom_{T}\left(X,E_{T}(k)\right)\right) \cong \Hom_{T}\left(H_{-i}(X),E_{T}(k)\right)= H_{-i}(X)^{\vee}$$
is an artinian $T$-module. By Lemma \ref{3.2} (i), $\fd_{S}(X^{\vee})= \id_{S}(X) < \infty$, so by \cite[Theorem 4.9]{DGI}, $X^{\vee}$ is virtually small over $S$ in the sense of \cite[4.1]{DGI}. On the other hand, $\fd_{R}(X^{\vee})= \id_{R}(X) < \infty$, so \cite[Theorem 5.5]{DGI} implies that $\fd_{R}(S)< \infty$.
\end{proof}

\begin{corollary} \label{3.15}
Let $f:(R,\mathfrak{m}) \rightarrow (S,\mathfrak{n})$ be a local homomorphism of noetherian local rings, and $M$ a non-zero finitely generated $S$-module. Then the following assertions hold:
\begin{enumerate}
\item[(i)] If $\id_{S}(M)< \infty$, then we have:
$$\edim(S)+ \id_{R}(M) \geq \edim(R)+ \id_{S}(M)$$
Moreover, if equality holds, then $f$ is an exceptional complete intersection map.
\item[(ii)] If $f$ is an exceptional complete intersection map, then we have:
$$\edim(S)+ \id_{R}(M) \geq \edim(R)+ \id_{S}(M)$$
\end{enumerate}
\end{corollary}

\begin{proof}
(i): If $\id_{R}(M)= \infty$, then there is nothing to prove, so we may assume that $\id_{R}(M)< \infty$. As a result, Lemma \ref{3.14} yields $\fd_{R}(S)< \infty$.

We first handle the special case where $f$ is surjective. Then $M$ is a finitely generated $R$-module as well as $S$-module, $S$ is a finitely generated $R$-module, $\id_{R}(M)< \infty$, $\id_{S}(M)< \infty$, and $\pd_{R}(S)=\fd_{R}(S)< \infty$. These assumptions allow us to take advantage of Auslander-Buchsbaum and Bass formulas (see \cite[Theorems 1.3.3 and 3.1.17]{BH}) to write:
\begin{equation*}
\begin{split}
 \id_{R}(M) - \id_{S}(M) & = \depth(\mathfrak{m},R) - \depth(\mathfrak{n},S) \\
 & = \depth(\mathfrak{m},R) - \depth_{R}(\mathfrak{m},S) \\
 & = \pd_{R}(S)
\end{split}
\end{equation*}
Therefore, it suffices to show that $\pd_{R}(S)\geq \edim(R) - \edim(S)$. To this end, we argue by induction on $\edim(R) - \edim(S)$.

If $\edim(R) - \edim(S)=0$, then there is nothing to prove. Now suppose that $\edim(R) - \edim(S)\geq 1$. As in the proof of Theorem \ref{3.12}, we can obtain a non-zerodivisor $a\in \Ker(f)\backslash \mathfrak{m}^{2}$. In light of Lemma \ref{3.7} (i) and the induction hypothesis applied to $\bar{f}:R/ \langle a \rangle \rightarrow S$, we see that:
\begin{equation*}
\begin{split}
 \pd_{R}(S) & = \pd_{R/ \langle a \rangle}(S)+1 \\
 & \geq \edim\left(R/ \langle a \rangle\right)- \edim(S)+1 \\
 & = \edim(R) - \edim(S)
\end{split}
\end{equation*}
Now if equality holds above, i.e. $\fd_{R}(S)= \pd_{R}(S)= \edim(R) - \edim(S)$, then \cite[Theorem 3.1]{BIK} implies that $f$ is an exceptional complete intersection map.

Next consider the general case. It is clear that $M\otimes_{S}\widehat{S}$ is a non-zero finitely generated $\widehat{S}$-module. By Corollary \ref{3.5}, we have $\id_{\widehat{S}}\left(M\otimes_{S}\widehat{S}\right) = \id_{S}(M)$. In addition, we observed in the proof of Theorem \ref{3.12} that $\id_{R}\left(M\otimes_{S}\widehat{S}\right)= \id_{R}(M)$ and $\edim\left(\widehat{S}\right)=\edim(S)$. Hence we can replace $S$ with $\widehat{S}$ and $M$ with $M\otimes_{S}\widehat{S}$, and assume accordingly that $S$ is $\mathfrak{n}$-adically complete. Then by Remark \ref{2.3}, $f$ has a Cohen factorization
$$R \xrightarrow{\dot{f}} R' \xrightarrow{f'} S$$
in which $\dot{f}$ is flat with regular closed fiber $R'/ \mathfrak{m}R'$ and $f'$ is surjective. By Lemmas \ref{3.11} and \ref{2.2}, we have:
\begin{equation*}
\begin{split}
 \id_{R'}(M) & \leq \id_{R}(M) + \edim\left(R'/ \mathfrak{m}R'\right) \\
 & = \id_{R}(M) + \edim(R')-\edim(R)
\end{split}
\end{equation*}
Applying the special case to $f'$, we get:
\begin{equation*}
\begin{split}
 \id_{R}(M)- \id_{S}(M) & \geq \id_{R'}(M)- \id_{S}(M) + \edim(R) - \edim(R') \\
 & \geq \edim(R')- \edim(S) + \edim(R) - \edim(R') \\
 & = \edim(R)-\edim(S)
\end{split}
\end{equation*}
Now if
$$\id_{R}(M)- \id_{S}(M) = \edim(R)-\edim(S),$$
then the above display shows that
$$\id_{R'}(M)- \id_{S}(M) = \edim(R')-\edim(S),$$
so by the special case, $f'$ is exceptional complete intersection, thereby Remark \ref{2.3} implies that $f$ is exceptional complete intersection.

(ii): If $\id_{R}(M)= \infty$, then there is nothing to prove, so we may assume that $\id_{R}(M)< \infty$. Since $f$ is exceptional complete intersection, \cite[Corollary 4.2]{ILP} yields
$$\id_{S}(M)= \id_{R}(M)- \dim(R) + \dim(S)< \infty.$$
Now the result follows from (i).
\end{proof}

\section*{Acknowledgement}
It is a pleasure to express my sincerest gratitude to professor Srikanth Iyengar for his invaluable comments and suggestions on this manuscript which is part of my Ph.D. thesis at Clemson University.


\end{document}